\documentclass{article}
\usepackage{fullpage}

\usepackage{amssymb}
 \usepackage{amsthm}

\usepackage{amsmath,amsfonts,bm}
\usepackage{graphicx}
\usepackage{lineno}
\usepackage{environ}
\usepackage{hyperref}
\usepackage{xcolor}

\usepackage{enumitem}
\setlist{  
	listparindent=1.5em}

\newcounter{quote}

\NewEnviron{myquote}{
	\vspace{0.2cm}%\par
	\refstepcounter{quote}%
	\parbox{15cm\textwidth0cm}%
	{\small\textit{\BODY}}%
	\vspace{0.2cm}}

\hypersetup{
	colorlinks=true, 
	linkcolor=blue,
	citecolor=blue,
}

\newtheorem{thm}{Theorem}
\newtheorem{theorem}{Theorem}[section]

\newtheorem{cor}[theorem]{Corollary}
\newtheorem{lemma}[theorem]{Lemma}

\theoremstyle{definition}

\theoremstyle{remark}

\newcommand{\id}{\mathrm{id}}

\newcommand{\aut}{\operatorname{Aut}}

\newcommand{\mc}{\mathfrak{c}}

\begin{document}

  \title{On the Automorphism Group of Token Graphs of Complete Bipartite Graphs}
  \author{Ruy Fabila-Monroy \thanks{Departamento de Matemáticas, Cinvestav. Supported by CONACYT FORDECYT-PRONACES/39570/2020.  {\tt  ruyfabila@math.cinvestav.edu.mx}}  \and
Ana Laura Trujillo-Negrete \thanks{Centro de Modelamiento Matemático (CNRS IRL2807), Universidad de Chile, Santiago, Chile. Supported by CONACYT(Mexico), Convocatoria 2021 de Estancias Posdoctorales por M\'exico en Apoyo por
	SARS-CoV-2(COVID-19), by ANID/Fondecyt Postdoctorado 3220838, and by ANID Basal Grant CMM FB210005. {\tt ltrujillo@dim.uchile.cl} }}

\maketitle

\begin{abstract}
Let $G$ be a graph of order $n$ and let $k\in \{1,2,\ldots,n-1\}$. 
	The $k$-token graph of $G$ is the graph, $F_k(G)$, whose vertices are all the $k$-subsets of vertices
	of $G$, where two such $k$-sets are adjacent whenever their symmetric difference is an edge of $G$. 
	In this paper, we determine the automorphism group of the $k$-token graph of the complete bipartite graph
	$K_{m,n}$. 

\end{abstract}

\section{Introduction}

Throughout this paper, $G$ is a simple graph of order $n\geq 2$ and $k\in \{1,2,\ldots,n-1\}$. 
The \emph{$k$-token graph}  of $G$ is the graph, $F_k(G)$, whose vertices
are all the $k$-subsets of vertices of $G$; where two vertices $A$ and $B$ of $F_k(G)$ are adjacent whenever their
symmetric difference, $A \triangle B$, is a pair of adjacent vertices in $G$. See Figure~\ref{fig:example_tokengraph} for an example.
As far as we know, Token graphs have been defined, independently (under different names), at least four times \cite{distance,double_1,rudolph,ruy_tokens}. The name ``Token graphs'' was given in \cite{ruy_tokens}, and
it is inspired  by the following interpretation. Suppose that $k$
indistinguishable tokens are placed on the vertices of $G$ with at most one token per vertex.
Form a new graph, $F_k(G)$,
whose vertices are all the possible token configurations, where two token configurations are adjacent
if one can be obtained from the other by taking a token and sliding it along an edge to an unoccupied vertex. We use this interpretation throughout this paper. 

\begin{figure}[t]
	\centering 
	\includegraphics[width=0.55\textwidth]{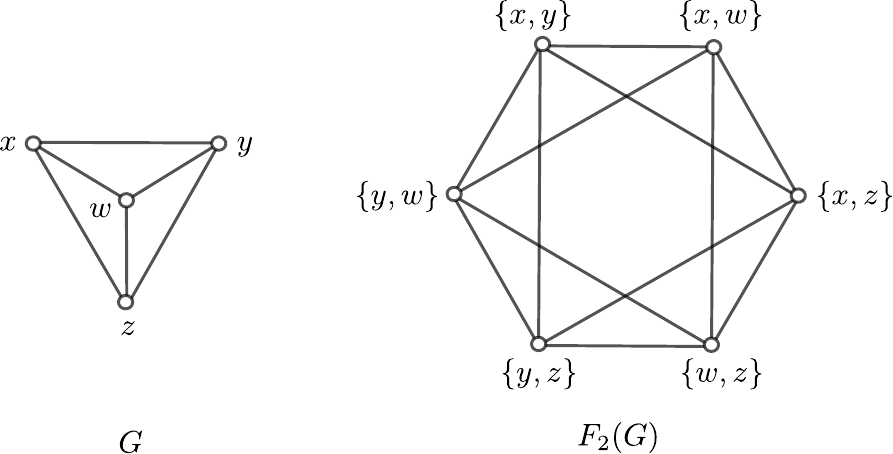}
	\caption{A graph $G$ and its $2$-token graph $F_2(G)$.}
	\label{fig:example_tokengraph}
\end{figure}

Token graphs have several connections with other combinatorial objects. The $k$-token graph of the complete graph, $K_n$, is precisely the Johnson graph, $J(n,k)$; the Johnson graph has been widely studied due to connections with coding theory \cite{etzion,ganesan}. Token graphs have been used to model some physical phenomena, see for example \cite{godsil,alzaga,yingkai,barghi,rudolph}; in error correcting codes, the packing number of the $2$-token graph of the path graph of $n$ vertices, corresponds to the size of the largest code of length $n$ and constant weight $2$ that can correct a single adjacent transposition, see \cite{gomez}; Token graphs also have connections with some reconfiguration problems, see \cite{demaine,yama}.

An \emph{automorphism} of $G$ is a bijection $\varphi: V(G) \to V(G)$ such that $x$ is adjacent
to $y$ if and only if $\varphi(x)$ is adjacent to $\varphi(y)$. The set of all automorphisms of $G$
forms a group under function composition; we denote it with $\aut(G)$.
The study of combinatorial and algebraic properties of Token graphs has often followed the following approach.

\begin{myquote} Given a graph invariant on $G$, what can be said about the same graph invariant on $F_k(G)$? 
\end{myquote}

\noindent See, for example, \cite{double_1,token2,ruy_tokens,trees_connectivity,jacob,leanos2019edgeconnectivity,leanos2018connectivity,Rivera2018}. In this paper, we follow this approach with regards to the automorphism group of a graph. 

Let $\varphi \in \aut(G)$, and let $\iota(\varphi)$ be the function
that maps every $A \in V(F_k(G))$ to \[\iota(\varphi)(A):=\{\varphi(v) : v \in A\}.\]
In~\cite{rivera-ibarra}, it is shown that $\iota$ is an injective group homomorphism.
Thus, 
\begin{equation}\label{eq:G<F}
	\aut(G) \le \aut(F_k(G)).
\end{equation}

Let $\mathfrak{c}$ be the map that sends
every set of $k$ vertices of $G$ to its complement in $V(G)$. This
map is an isomorphism from $F_k(G)$ to $F_{n-k}(G)$. If $k=n/2$, then
$\mathfrak{c}$ is an automorphism of $F_k(G)$, which we call the 
\emph{complement automorphism}. When $k=n/2$, we have that
$\mathfrak{c} \notin \iota(\aut(G))$~\cite{rec}, and
\[\iota(\varphi)\circ \mathfrak{c}=\mathfrak{c}\circ \iota(\varphi),\]
for all $\varphi \in \aut(G)$. Thus, in the case when $k=n/2$,
we have that
\begin{equation} \label{eq:G<F_n/2}
	\aut(G)\times \mathbb{Z}_2 \le \aut (F_k(G)).
\end{equation}

There are examples where inequalities (\ref{eq:G<F}) and  (\ref{eq:G<F_n/2}) are
tight and examples where they are strict. 
When either of these inequalities is tight it has been shown~\cite{rec} that a graph $H\simeq F_k(G)$ can be
described as the token graph of $G$ in a unique way (modulo automorphisms of $G$).
In this case we say that $F_k(G)$
is \emph{uniquely reconstructible as the $k$-token graph} of $G$. 
So far, the families of graphs for which $Aut(F_k(G))$ has been studied are:
the complete graph, see \cite{automorphism_johnsongraph, ganesan, ramras}, the path graph for $k\neq n/2$, cycle, star, fan and wheel graphs for $k=2$ \cite{rivera-ibarra}, and
($C_4$, Diamond)-free connected graphs~\cite{rec}. For all of these families either (\ref{eq:G<F}) or  (\ref{eq:G<F_n/2})
is tight. In this paper, we study the automorphism group of token graphs of complete bipartite graphs.
Specifically, we prove that $F_k(K_{m,n})$ is uniquely reconstructible if and only if $m \ne 2$ or $k\in \{1,m+n-1\}$.

Given a set $S$, let $2^S$ be denote the family of all subsets of $S$. Thus, $|2^S|=2^{|S|}$.
Let $[n]:=\{1,\dots,n\}$, and let $\binom{[n]}{r}$ be the set of all subsets of $[n]$ of $r$ elements. 
The symmetric group $S_n$ acts on $\binom{[n]}{r}$ as follows. For a permutation $\pi \in S_n$ and $X \in \binom{[n]}{r}$, let $\pi(X):=\{\pi(x): x \in X\}$. In what follows we regard $\binom{[n]}{r}$ as an $S_n$-set with this action. 
Our main result is the computation of the automorphism group of tokens graphs of complete bipartite graphs.
\begin{thm}\label{thm:bipartite} 
	Let $m$ and $n$ be positive integers with $m \le n$. Then
	\ {}
	\begin{itemize}
		\item[a)]\[\aut(F_2(K_{2,2}))=\mathbb{Z}_2 \times S_4;\]
		\item[b)] for all $n >2$ and $1 < k< n+1$,
		\[ \aut(F_k(K_{2,n})) = \begin{cases} 
			\mathbb{Z}_2 \wr_{\binom{[n]}{k-1}} S_n & \textrm{ if } k\neq \frac{n+2}{2}, \\
			\mathbb{Z}_2 \wr_{\binom{[n]}{k-1}} (S_n\times \mathbb{Z}_2)  & \textrm{ otherwise; and} \\
		\end{cases}
		\]

		\item[c)] for all $m \neq 2$,
		\[ \aut(F_k(K_{m,n}))=\begin{cases}
			\aut(K_{m,n}) & \textrm{ if } k\neq \frac{n+m}{2}, \\
			\aut(K_{m,n}) \times \mathbb{Z}_2  & \textrm{ otherwise.}
		\end{cases}\]

	\end{itemize}
	
\end{thm}\qed

We note that
\[
Aut(K_{m,n}) = \begin{cases}
	S_m \times S_n & \text{if $m\neq n$,} \\
	S_m \times S_n \rtimes \mathbb{Z}_2 & \text{if $m=n$.}
\end{cases}
\]
Indeed, the automorphisms of $K_{m,n}$ are generated by permutations within each vertex class, resulting in the group $S_m \times S_n$ when $m \neq n$. When $m = n$, there is also an automorphism that swaps the vertex classes, generating a subgroup isomorphic to $\mathbb{Z}_2$. Thus, in this case, the automorphism group is $S_m \times S_n \rtimes \mathbb{Z}_2$.

As a result, Theorem~\ref{thm:bipartite} implies that
\[|\aut(F_k(K_{2,n}))|=\begin{cases}
	2^{\binom{n}{k-1}-1}|\aut(K_{2,n})| & \text{ if $k \neq \frac{n+2}{2}$, and } \\[10pt]
	2^{\binom{n}{k-1}-1}|\aut(K_{2,n})\times \mathbb{Z}_2| & \text{ otherwise.}
\end{cases}\]
The particular cases of  $F_2(K_{n,n})$, $F_k(K_{1,n})$ and $F_{(n+2)/2}(K_{2,n})$ in Theorem~\ref{thm:bipartite},
were proven recently in~\cite{transitive}.

Before continuing,  we provide some  notation used throughout this paper. Given a graph $H$ and a vertex
$v\in H$, let $N(v)$ and $\deg(v)$ denote the neighborhood and degree of $v$ (in $H$), respectively. Given two vertices $u,v\in V(H)$, let $d(u,v)$ denote
the distance between vertices $u$ and $v$. 

\section{Overview}

The proof begins by establishing a lower bound on the number of automorphisms in Section~\ref{section:inclusions}. In Section~\ref{section:equality}, we show that this lower bound is, in fact, the exact number of automorphisms, thereby providing both a lower and an upper bound. This structure allows us to fully characterize the automorphisms of $F_k(K_{m,n})$.

In Section~\ref{section:inclusions}, we exhibit certain subgroups of $\aut(F_k(K_{m,n}))$. Specifically, when~$m \neq 2$, these subgroups are generated by the induced automorphisms of $K_{m,n}$, and, additionally, by the complement automorphism when $k=\frac{n+m}{2}$. For $m = 2$, in addition to the induced and complement automorphisms (when applicable), we identify new automorphisms of $F_k(K_{2,n})$. We also show that the subgroup generated by these new automorphisms, together with the induced and complement automorphisms (for~$k = \frac{n+2}{2}$), forms the group described in Theorem~\ref{thm:bipartite}.

In Section~\ref{section:equality}, we prove that the automorphisms identified in Section~\ref{section:inclusions} are the only automorphisms that exist, thus providing an exact upper bound matching the lower bound. Let $X$ and $Y$ be the parts of $K_{m,n}$, with $m = |X| \leq |Y| = n$, and let $\varphi$ be a fixed automorphism of $F_k(K_{m,n})$. The vertex set of $F_k(K_{m,n})$ is partitioned into subsets ${H_1, \dots, H_r}$, where each $H_i$ consists of vertices with exactly $i$ tokens on $X$.

In Lemma~\ref{lemma:H_0-image}, we establish the action of $\varphi$ on ${H_1, \dots, H_r}$, proving that $\varphi$ must map each $H_i$ either to itself or, in specific cases, to $H_{r-i}$. Next, Lemmas~\ref{lem:PQ} and~\ref{lem:permutation_sigma} characterize how $\varphi$ acts on the vertices of $H_0$. Corollary~\ref{cor:H_0-labelling} then summarizes these results, detailing the image of each vertex in $H_0$ based on certain permutations, thus specifying the number of possibilities for $\varphi(H_0)$.

We continue by analyzing the image of vertices in $H_1$ under $\varphi$. For $m = 2$, Lemma~\ref{lemma:H_0-image} proves the existence of a function that, together with the image of vertices in $H_1$ under $\varphi$, determines the image of vertices in $H_1$ under $\varphi$.  For $m > 2$, Lemmas~\ref{lem:m>2} and~\ref{lem:m>2_A} describe how $\varphi$ extends from $H_0$ to $H_1$. In Lemma~\ref{lem:H_i}, we show that once the action of $\varphi$ on $H_0 \cup H_1$ is determined, the image of vertices in $H_2 \cup \dots \cup H_r$ is uniquely determined. Finally, Lemma~\ref{lem:bipartite_upper_bound} establishes the exact number of possibilities for $\varphi$, thus confirming that no additional automorphisms exist beyond those described in Section~\ref{section:inclusions}.

\section{Preliminaries}
\label{section:preliminaries}

In this section, we present some known facts and definitions about group products. 
Let $G$ and $H$ be groups. The \emph{Cartesian} or \emph{direct product} of $G$ and $H$ is the group,
$G\times H$,  whose elements are all the tuples $(g,h)$, where $g \in G$
and $h \in H$, and with operation
\[(g,h)(g',h')=(gg',hh').\]
This definition can be generalized for an arbitrary number of factors.
Let $\{G_w\}_{w \in \Omega}$ be a family of groups. The Cartesian or direct product
of this  family is the group $\Pi_{w \in \Omega} G_w$ whose elements are all the tuples $a=(a_w)_{w \in \Omega}$,
and with operation
\[ab=(a_wb_w)_{w\in \Omega},\]
for $a,b \in \Pi_{w \in \Omega} G_w$. If all the $G_w$ are isomorphic to $G$, then we
may also denote  $\Pi_{w \in \Omega} G_w$ with $G^\Omega$. 

A \emph{homomorphism} between $G$ and $H$ is a function $\varphi:G \to H$
such that \[\varphi(xy)=\varphi(x)\varphi(y),\] for all $x,y \in G$.
A bijective homomorphism of $G$ with itself is called an \emph{automorphism} of $G$.
The set of all automorphisms of $G$ forms a group under function composition; we
denote it with $\aut(G)$. Let $\varphi:H \to \aut(G)$ be a group homomorphism.
Let $G \rtimes_\varphi H$ be the group whose elements are all the tuples $(g,h)$, where $g \in G$
and $h \in H$, and with operation
\[(g,h)(g',h')=(g\varphi(h)(g'),hh').\]
We call $G \rtimes_\varphi H$ an \emph{(outer) semidirect product} of $G$ and $H$. Note that
the structure of a semidirect product of $G$ and $H$ does not only depend on $G$ and $H$
but on $\varphi$ as well. Occasionally the $\varphi$ is dropped from the notation and
$G \rtimes_\varphi H$ is simply written as $G \rtimes H$.

Alternatively, suppose that $G$ and $H$ are subgroups of a larger group $N$, in which $G$ is normal.
Suppose that $G \cap H=\{e\}$, and that $GH=N$. Then the function $\varphi: H \to \aut(G)$,
defined by $\varphi(h)(x):=hxh^{-1}$ is a group homomorphism, and \[N \simeq G \rtimes_\varphi H.\]
In this case we say that $N$ is an \emph{(inner) semidirect product} of $G$ and $H$.

Let $\Omega$ be a set. An {\em action} of $H$ in $\Omega$ is a function $\alpha:H\times \Omega \to \Omega$, denoted by
$\alpha:(g,x) \mapsto gx$, such that:
\begin{itemize}
	\item $ex =x$ for all $x \in \Omega$ ,where $e$ is the identity of $H$; and
	\item $g(hx)=(gh)x$ for all $g,h \in H$ and $x \in \Omega$.
\end{itemize}
In this case we call $\Omega$ a \emph{$H$-set}, and we say that $H$ \emph{acts} on $\Omega$.
The action of $H$ on $\Omega$ induces a group homomorphism $\varphi$ from $H$ to $\aut(G^\Omega)$, by
letting
\[\varphi(h)((x_{w})_{w \in \Omega}):=(x_{h^{-1}w})_{w \in \Omega}.\]
We define the \emph{wreath product},  $G \wr_\Omega H$, as the group $G^\Omega \rtimes_\varphi H$.

\section{Proof of Theorem~\ref{thm:bipartite}}
\label{section:bipartite}

We have that $F_2({K_{2,2}}) \simeq K_{2,4}.$ 
Thus, \[\aut (F_2(K_{2,2}))=\mathbb{Z}_2 \times S_4.\] 
This proves $a)$ of Theorem~\ref{thm:bipartite}.
For the remainder of this section assume that $n>2$. In~\cite{rec}, it was shown that $F_k(K_{1,n})$
is uniquely reconstructible as the $k$-token graph of $K_{1,n}$. Thus,
\[\aut(F_k(K_{1,n})) \simeq \begin{cases}
	\aut(K_{1,n})  &\text{ if } k \neq \frac{n+1}{2},\\
	\aut(K_{1,n}) \times \mathbb{Z}_2 &\text{ otherwise. }
\end{cases}
\]
For the remainder of this section assume that $m \ge 2$. 

Throughout this section, let  $\{X,Y\}$ be the bipartition of $K_{m,n}$
with $m=|X|$ and $n=|Y|$.
Let $r:=\min\{m,k\}$. For $i\in \{0,\ldots,r\}$, let
\begin{equation*} H_i:=\{A\in V(F_k(K_{m,n})):|A\cap X|=i\}.\end{equation*}
Let us establish some basic properties of subsets $H_i$'s. First, we note that   $H_0,H_1,\ldots,H_r$ are pairwise disjoint, and that each $H_i$ is an independent set of $F_k(K_{m,n})$. Now, consider a vertex $A\in H_i$ and let $B$ be a neighbor of $A$.  
Observe that $B\in H_{i-1}\cup H_{i+1}$, 
because either $|B\cap X|=|A\cap X|-1$ or $|B\cap X|=|A\cap X|+1$. 
Moreover, let us show that  $F_k(K_{m,n})$ is bipartite with bipartition  $\{\mathcal{B},\mathcal{R}\}$,  where
\[\mathcal{B}:=\bigcup_{i \text{ even}} H_i \quad \textrm{and}\quad \mathcal{R}:=\bigcup\limits_{i\text{ odd}} H_i. \]
Indeed, consider an edge $AB$ of $F_k(K_{m,n})$, and let $ab$ be the edge of $K_{m,n}$ with $A\setminus B=\{a\}$ and $B\setminus A=\{b\}$.
Without any loss of generality we may assume that $a\in X$ and $b\in Y$. 
Observe that $A\cap X=(B\cap X)\cup \{a\}$ and so, 
the parity of $|A\cap X|$ and $|B\cap X|$ is different, implying that either $A\in\mathcal{B}$ and $B\in \mathcal{R}$, or $A\in \mathcal{R}$
and $B\in \mathcal{B}$. Thus, $\{\mathcal{B},\mathcal{R}\}$ is the bipartition of $F_k(K_{m,n})$. 

\subsection{Inclusions}
\label{section:inclusions}

First, we show that the groups stated in $b)$ and $c)$ of Theorem~\ref{thm:bipartite} are indeed subgroups
of $\aut(F_k(K_{m,n}))$.

\bigskip

\noindent \bm{$c)$}
This follows from the fact that if $k \neq \frac{n+m}{2}$, then $\iota(\aut(K_{m,n})) \le \aut(F_k(K_{m,n}))$;
and if $k = \frac{n+m}{2}$, then $\iota(\aut(K_{m,n})) \times \mathbb{Z}_2 \le \aut(F_k(K_{m,n}))$.

\bigskip

\noindent \bm{$b)$} Suppose that $m=2$ and let
\[Y:=\{1,\dots,n\}.\]
For $x \in X$, let $\overline{x}$ be the only
vertex in $X \setminus \{x\}.$
For a vertex $A \in H_1$ let \[\overline{A}:=(A\setminus X) \cup (X \setminus A).\]
That is, $\overline{A}$ is obtained by replacing the only element of $X$ in $A$ with the other
element of $X$. 

Let $\alpha \in 2^{\binom{[n]}{k-1}}$, and let $\varphi_\alpha: V(F_k(K_{2,n})) \to  V(F_k(K_{2,n}))$ be the function
defined by
\[\varphi_\alpha(A):= \begin{cases}
	\overline{A} & \text{ if $A \cap Y \in \alpha$, and } \\
	A   & \text{ otherwise. }
\end{cases}\]
Observe that $\alpha$ corresponds to a collection of $(k-1)$-subsets in $[n]$, and so, for each $S\in \alpha$, we have that $\varphi_\alpha$ is the function defined by $\varphi_\alpha(S\cup\{x\})=S\cup\{\overline{x}\}$. 

We recall that the group $\mathbb{Z}_2^{\binom{[n]}{k-1}}$ is the direct (or Cartesian) product of $\binom{n}{k-1}$ copies of $\mathbb{Z}_2$, where $\mathbb{Z}_2$ is the cyclic group of order $2$. 

We claim that

\begin{myquote} $\varphi_\alpha \in \aut(F_k(K_{2,n}))$, and the subgroup of $\aut(F_k(K_{2,n}))$
	generated by the $\varphi_{\alpha}$ is isomorphic to $\mathbb{Z}_2^{\binom{[n]}{k-1}}$.
\end{myquote}

It is straightforward to show that $\varphi_\alpha$ is bijective. We now show that $\varphi_\alpha$ is an automorphism. Let $AB$ be an edge of $F_k(K_{2,n})$.
Without loss of generality assume that $A \in H_1$; thus, $B \in H_0 \cup H_2$.
This implies that $\varphi_\alpha(B)=B$. If $A \cap Y \notin \alpha$,
then $\varphi_{\alpha}(A)=A$, and $\varphi_\alpha(A)$ is adjacent to $\varphi_\alpha(B)$ in this case.
Suppose that $A \cap Y \in \alpha$. Thus, $\varphi_\alpha(A)=\overline{A}$. Let $\{x,b\}:=A \triangle B$,
with $x \in X$ and $b \in Y$.
Note that regardless of whether $B \in H_0$ or $B \in H_2$, we have that
$\varphi_\alpha(A) \triangle \varphi_\alpha(B)=\{\overline{x}, b\}.$ Thus, $\varphi_\alpha(A)$
is adjacent to $\varphi_\alpha(B)$, and $\varphi_\alpha$ is an automorphism of 
$F_k(K_{2, n})$. 
Let us now show that for any two subsets $\alpha,\beta\in 2^{\binom{[n]}{k-1}}$, we have 
\begin{equation}
	\label{eq:dif-sim}
	\varphi_\alpha \varphi_\beta=\varphi_{\alpha \triangle \beta}.
\end{equation}
Indeed, consider a vertex $A\in F_k(K_{2,n})$. If $|A\cap Y|\ne k-1$, then $\varphi_\gamma(A)=A$ for each $\gamma\in \{\alpha,\beta,\alpha\triangle \beta\}$, and so~\eqref{eq:dif-sim} holds trivially. Suppose now that $|A\cap Y|=k-1$. Let
$S:=A\cap Y$. If $S\notin \alpha\triangle \beta$, it is straightforward to see that 
\[\varphi_\alpha(\varphi_\beta(A))=A=\varphi_{\alpha\triangle\beta}(A),\]
as desired. On the other hand, if $S\in \alpha\triangle \beta $, then 
\[\varphi_\alpha(\varphi_\beta(A))=\overline{A}=\varphi_{\alpha\triangle\beta}(A).\]
Thus,~\eqref{eq:dif-sim} holds.
The fact that 
$N_1:=\left \langle \left \{\varphi_\alpha: \alpha \in 2^{\binom{[n]}{k-1}} \right \} \right \rangle$ is
isomorphic to $\mathbb{Z}_2^{\binom{[n]}{k-1}}$, comes from considering $\alpha$ as
an element of $\mathbb{Z}_2^{\binom{[n]}{k-1}}$ (as explained above)  and \eqref{eq:dif-sim}. 
Note that $\varphi_\alpha$ is equal to the identity if and only if $\alpha=\emptyset$. Indeed, $\varphi_\alpha(A)=A$ for all $A\in V(F_k(K_{2,n}))$ if and only if $A\cap Y\notin \alpha$ for every $A$, which occurs precisely when $\alpha=\emptyset$.

Let $\pi \in S_n$. Let $\phi_\pi: V(K_{2,n}) \to  V(K_{2,n})$ be the function defined by
\[\phi_\pi(u):=\begin{cases}
	\pi(u) & \text{ if $u \in Y$, and } \\
	u   & \text{ otherwise. }
\end{cases}\]  Note that $\phi_\pi \in \aut(K_{2,n})$.
Let \[\psi_\pi:=\iota(\phi_\pi).\] Thus, $N_2:=\left \langle \left \{\psi_\pi: \pi \in S_n  \right \} \right \rangle$ is
a subgroup of $F_k(K_{2,n})$ isomorphic to $S_n$. Suppose that $\alpha \neq \emptyset$ and let $A \in V(F_k(K_{2,n}))$ such that
$A \cap Y \in \alpha$. Observe that $|A\cap X|=1$. Let  $\{x\}:=A\cap X$. We have that $\varphi_\alpha(A)=\overline{A}$, and so $x\notin \varphi_\alpha(A)$. On the other hand, for every $\pi \in S_n$, we have 
$x\in \psi_\pi(A)$, and so,
\[\varphi_\alpha(A) \neq \psi_\pi(A).\]
Thus, \[N_1 \cap N_2 = \{\id \}.\] Let $\pi(\alpha):=\{\pi(S):S \in \alpha\}.$ 
Let us now show that 
\begin{equation}
	\label{eq:comp-pi}
	\psi_\pi \varphi_\alpha \psi_\pi^{-1}=\varphi_{\pi(\alpha)}.
\end{equation}
Consider a vertex $A\in F_k(K_{2,n})$, and let $S:=A\cap Y$. We have that
\[\varphi_{\pi(\alpha)}(A)=\begin{cases}
	\overline{A} & \textrm{if $S\in \pi(\alpha)$,}\\
	A & \textrm{otherwise}. 
\end{cases}\]
If $|S|\ne k-1$, then $\varphi_{\alpha}(A)=A$ and $\varphi_{\pi(\alpha)}(\psi_{\pi}(A))=\psi_{\pi}(A)$, 
and so \eqref{eq:comp-pi} holds. 

Suppose now that $|S|=k-1$. Let $\{x\}:=A\cap X$. 
We have
\[\varphi_\alpha\psi_\pi^{-1}(A)=\varphi_\alpha(\{x\}\cup \pi^{-1}(S))=\begin{cases}
	\{\overline{x}\}\cup \pi^{-1}(S) & \textrm{if $\pi^{-1}(S)\in \alpha$}, \\
	\{x\}\cup \pi^{-1}(S)& \textrm{otherwise},
\end{cases}\]
and then, 
\[\psi_\pi\varphi_\alpha\psi_\pi^{-1}(A)=\begin{cases}
	\{\overline{x}\}\cup S =\overline{A} & \textrm{if $\pi^{-1}(S)\in \alpha$}, \\
	\{x\}\cup S=A& \textrm{otherwise}.
\end{cases}\]
Since $S\in \pi(\alpha)$ if and only if $\pi^{-1}(S)\in \alpha$, we conclude that~\eqref{eq:comp-pi} holds.

If $k\ne \frac{n+2}{2}$, then $N_1N_2$ is a subgroup of $\aut(F_k(K_{2,n}))$ isomorphic to \[\mathbb{Z}_2 \wr_{\binom{[n]}{k-1}} S_n.\]

Suppose now that $k=\frac{n+2}{2}$. In this case, $\mathfrak{c}$ is an automorphism of $F_k(K_{2,n})$.
Let $N_1$ and $N_2$ be as above, and let $N_3:=\langle N_2\cup \{\mc\}\rangle$.  
Observe that $\mc(H_0)=H_2$ and $\psi_\pi(H_0)=H_0$, for each $\psi_\pi\in N_2$, so $\mc\notin N_2$, and then 
$N_3$ is a subgroup of $\aut(F_k(K_{2,n}))$. Further, it is straightforward to see that 
$\mc^{-1}=\mc$ and 
$\mathfrak{c}\psi_\pi=\psi_\pi\mathfrak{c}$, so we have $N_3\simeq S_n\times \mathbb{Z}_2$. 
Moreover, since $\mc\notin N_1$, it follows that $N_1\cap N_3=\{\id\}$.

Given a set $S\subset Y$ with $|S|=k-1$, let $\overline{S}:=Y\setminus S$. 
Let $\mc(\alpha):=\{\overline{S}:S\in \alpha\}$. Let us now show that 
\begin{equation}
	\label{eq:composition-complement}
	\mc\varphi_\alpha \mc^{-1}=\varphi_{\mc(\alpha)}.
\end{equation}
Fix a vertex $A\in F_k(K_{2,n})$. Observe that if $A\in H_0\cup H_2$, then 
\[\varphi_\alpha(A)=A=\varphi_{\mc(\alpha)}(A), \]
and so~\eqref{eq:composition-complement} holds trivially. Assume that $A\in H_1$. Let $\{x\}:=A\cap X$
and $S:=A\cap Y$. We have 
\[\varphi_\alpha\mc^{-1}(A)=\varphi_\alpha\left(\{\overline{x}\}\cup \overline{S}\right)=\begin{cases}
	\{x\}\cup \overline{S} & \textrm{if $\overline{S}\in \alpha$}, \\
	\{\overline{x}\}\cup \overline{S} & \textrm{otherwise},
\end{cases}\]
and so, 
\[\mc\varphi_\alpha\mc^{-1}(A)=\begin{cases}
	\{\overline{x}\}\cup S=\overline{A} & \textrm{if $\overline{S}\in \alpha$}, \\
	\{x\}\cup S=A & \textrm{otherwise}.
\end{cases}\]
On the other hand, we have 
\[\varphi_{\mc(\alpha)}(A)=\begin{cases}
	\overline{A} & \textrm{if $S\in \mc(\alpha)$}, \\
	A & \textrm{otherwise}.
\end{cases}\]
Since $S\in \mc(\alpha)$ if and only if $\overline{S}\in \alpha$, \eqref{eq:composition-complement} is satisfied. 
Then, by~\eqref{eq:comp-pi}~and~\eqref{eq:composition-complement}, it follows that $N_1N_3$ is a subgroup of $\aut(F_k(K_{2,n}))$
isomorphic to 
\[\mathbb{Z}_2 \wr_{\binom{[n]}{k-1}} (S_n\times \mathbb{Z}_2).\]

\subsection{Equality}
\label{section:equality}

We now show that $\aut(F_k(K_{m,n}))$ has no more subgroups than those exhibited in the previous section. 
Let $\varphi$ be an arbitrary, but fixed, automorphism of $F_k(K_{m,n})$; and
let \[X:=\{x_1,\dots,x_m\} \textrm{ and } Y:=\{y_1,\dots,y_n\}.\]

\subsubsection*{The action of $\varphi$ on $\{H_0,\dots,H_r\}$} 
We first show that for all $i$ either $\varphi(H_i)  =H_i$ or $\varphi(H_i)  =H_{r-i}$; and give conditions
on when the latter can occur. 

\begin{lemma}
	\label{lemma:H_0-image}
	\-\ \begin{itemize}
		\item If $k\neq \frac{m+n}{2}$ and $m < n$, then 
		\[\varphi(H_i)=H_i,\]
		for all $i=0,\dots,r$; and
		
		\item if $k= \frac{m+n}{2}$ or $m = n$, then  either
		\[\varphi(H_i)  =H_i  \textrm{ or } \varphi(H_i)  =H_{r-i},\]
		for all $i=0,\dots,r$.
		\end {itemize}
	\end{lemma}
	\begin{proof}
		First we show that $\varphi(H_0)=H_0$ or $\varphi(H_0)=H_r$. 
		Let $A\in H_0$ and let  $B:=\varphi(A)$.  Since $\deg(A)=km$, we have that $\deg(B)=km$. 
		Let $0 \le j \le r$ be such that $B \in H_j$.
		Thus, $\deg(B)=j(n-k+j)+(k-j)(m-j)=km$; this implies that $j=0$ or $j=\frac{m-n+2k}{2}$, where in the last case, $j$ is an integer. 
		Suppose that $0 < j < r$. 
		Let $B_1\in H_{j-1}, B_2 \in H_{j+1}$ be neighbors of $B$.
		Since $j \neq 0$, we have that 
		\[\deg(B_1)=(j-1)(n-k+(j-1))+(k-(j-1))(m-(j-1))=km+n-m-2k+2 \] and 
		\[\deg(B_2)=(j+1)(n-k+(j+1))+(k-(j+1))(m-(j+1))=km-n+m+2k+2.\]
		Since all the neighbors of $A$ are in $H_1$, they all have the same degree.  
		Thus, $\deg(B_1)=\deg(B_2)$, which implies that $n-m=2k$. Since $j=\frac{m-n+2k}{2}$, we
		have that $j=0$, a contradiction. Thus, \[\varphi(H_0)=H_0 \textrm{ or } \varphi(H_0)=H_r.\]
		
		Suppose that $r=j=\frac{m-n+2k}{2}$. This implies that  if $r=k$, then $m=n$; and if $r=m$, then $k=\frac{m+n}{2}$. In particular, 
		if $k \neq \frac{m+n}{2}$ and $m \neq n$, then \[ \varphi(H_0)=H_0.\]
		Note that, for $0 \le i \le r$, $H_i$ is precisely the set of vertices that are at distance equal to $i$ from $H_0$ and at distance
		$r-i$ from $H_r$. Thus, if $\varphi(H_0)=H_0$, then $\varphi(H_i)=H_i$, and if  $\varphi(H_0)=H_r$, then $\varphi(H_i)=H_{r-i}$. 
		The result follows. 
	\end{proof}

	\subsubsection{The action of $\varphi$ on $H_0$} \label{sec:inc}
	
	We now characterize how $\varphi$ maps $H_0$ to $\varphi(H_0)$.	
	Suppose that $k < m =n$ or $m <n$. 
	For $S \subset Y$, let \[P_S:=\{A \in H_0:S \subset A \} \textrm{ and } \overline{P}_S:=\{A \in H_r:S \cap A=\emptyset \}\]
	and for $S \subset X$, let \[Q_S:=\{A \in H_r:S \subset A \}.\]
	\begin{lemma}\label{lem:PQ}
		Suppose that $k<m=n$ or $m<n$. Then for every  $S\subset Y$ with $1\leq |S|\leq k-1$,
		there exists a set $S'$ contained either in $X$ or in $Y$, with $|S'|=|S|$, such that
		\[\varphi(P_S)=\begin{cases}
			P_{S'}   & \text{ if $\varphi(H_0)=H_0$; } \\[3ex] 
			Q_{S'} & \text{ if $\varphi(H_0)=H_r$ and $k<m=n$; and  } \\[3ex] 
			\overline{P}_{S'}  & \text{ if $\varphi(H_0)=H_r$, $m<n$ and $k=\frac{m+n}{2}$. } 
		\end{cases}
		\] 
	\end{lemma}
	\begin{proof}
		By Lemma~\ref{lemma:H_0-image}, we know that $\varphi(H_0)=H_0$ or $\varphi(H_0)=H_r$. We now proceed by induction on $k-|S|$. 
		\begin{itemize}
			\item Suppose that $k-|S|=1$. Then $|S|=k-1$.

			Let $\{A_1,\ldots,A_p\}:=P_S$. For every $1 \le i \le p$, let $B_i:=\varphi(A_i)$.
			Let \[S'=\begin{cases}
				\bigcap_{i=1}^p B_i  & \text{ if $\varphi(H_0)=H_0$;} \\[3ex] 
				\bigcap_{i=1}^p B_i  & \text{ if $\varphi(H_0)=H_r$ and $k<m=n$; and  } \\[3ex]
				\bigcap_{i=1}^p ( (X \cup Y)\setminus B_i)  & \text{ if $\varphi(H_0)=H_r$, $m<n$ and $k=\frac{m+n}{2}$. } 
			\end{cases}
			\] 
			Note that 
			\[\varphi(P_S) \subset \begin{cases}
				P_{S'}   & \text{ if $\varphi(H_0)=H_0$; } \\[3ex] 
				Q_{S'} & \text{ if $\varphi(H_0)=H_r$ and $k<m=n$; and  } \\[3ex] 
				\overline{P}_{S'}  & \text{ if $\varphi(H_0)=H_r$, $m<n$ and $k=\frac{m+n}{2}$. } 
			\end{cases}
			\] 
			It remains to show that  $|S'|=|S|$. 
			
			\begin{itemize}
				\item Suppose that $k=n-1$. We have that $p=2$. If $\varphi(H_0)=H_0$,
				then $B_1, B_2 \subset Y$ and $|S'|=|B_1 \cap B_2|=n-2=k-1=|S|$. 
				If $\varphi(H_0)=H_r$ and $k<m=n$, then $B_1, B_2 \subset X$ and $|S'|=|B_1 \cap B_2|=n-2=k-1=|S|$. 
				Suppose that  $\varphi(H_0)=H_r$, $m<n$ and $k=\frac{m+n}{2}$. 
				We have that $S'=((X\cup Y)\setminus B_1) \cap ((X \cup Y) \setminus B_2)$. Note that $(X\cup Y)\setminus B_i \subset Y$
				and $|(X\cup Y)\setminus B_i|=n-(n-k)=n-1$. This implies that $|S'|=n-2=k-1=|S|$.
				
				\item Suppose $k<n-1$. Thus, $p> 2$. Note that $S\cup\{x_1\} \in H_1$ is adjacent to any vertex in  $\{A_1,\ldots,A_p\}$. 
				For a contradiction suppose that $|S'|<k-1$.
				\begin{itemize}
					\item Suppose that either $\varphi(H_0)=H_0$, or  $\varphi(H_0)=H_r$ and $k<m=n$.
					We have that $S'=\bigcap_{i=1}^p B_i$. Thus, there exist three indices $j,l,t$ such that 
					$B_j\cap B_l\neq B_l \cap B_t$; this implies that the vertices $B_j,B_l$ and $B_t$ have no common neighbors in $\varphi(H_1)$;
					which contradicts the assumption that $\varphi$ is an automorphism. Therefore, $|S'|=k-1=|S|$.
					
					\item Suppose that  $\varphi(H_0)=H_r$, $m<n$ and $k=\frac{m+n}{2}$.
					Since $m<n$ and $k=\frac{m+n}{2}$, we have that $k > m$ and $r<k$. In particular, these imply that $X \subset B_i$ and $|B_i \cap Y|=k-m=n-k$. 
					Therefore, $|(X \cup Y)\setminus B_i|=k$. 
					
					We now consider three vertices in $\{B_1,\dots,B_p\}$ such that their union contains at least $n-k+2$ vertices of $Y$, 
					which is feasible by the definition of $S'$, as  $|S'|\le k-2$, which implies $|Y\setminus S'|\ge n-k+2$. 
					Without loss of generality assume these three vertices are $B_1,B_2,B_3$. 
					Since these three vertices must have at least one common neighbor, we must have $|B_1\cap B_2\cap B_3|=k-1$; otherwise,
					there exist $i,j\in [3]$ such that $|B_i\triangle B_j|\ge 4$, implying $B_i$ and $B_j$ have no common neighbors, a contradiction. 
					Therefore, with $|B_1\cap B_2\cap B_3|=k-1$, any common neighbor of both $B_1$ and $B_2$ cannot also be
					a neighbor of $B_3$. This implies that the three vertices have no common neighbors in $\varphi(H_1)=H_{r-1}$, leading to a contradiction. 
					Therefore, $|S'|=k-1=|S|$.
				\end{itemize}
			\end{itemize}
			
			\item Suppose that $k-|S|\ge 2$, and assume that the lemma holds for  smaller values of $k-|S|$.
			
			Let $S_1, \ldots, S_{p}$ be the subsets of $Y$ such that $S \subset S_j$ and $|S_j|=|S|+1$.
			Note that $S=\bigcap_{j=1}^{p} S_j$ and $P_S=\bigcup_{j=1}^{p} P_{S_j}$.  
			By the induction hypothesis, there exists
			$S'_1,\ldots,S'_{p}$ contained either in $X$ or in $Y$, such that for every $1 \le j \le {p}$,
			\[\varphi(P_{S_j})=\begin{cases}
				P_{S'_j}   & \text{ if $\varphi(H_0)=H_0$; } \\[3ex] 
				Q_{S'_j} & \text{ if $\varphi(H_0)=H_r$ and $k<m=n$; and  } \\[3ex] 
				\overline{P}_{S'_j}  & \text{ if $\varphi(H_0)=H_r$, $m<n$ and $k=\frac{m+n}{2}$. } 
			\end{cases}
			\] 
			and $|S'_j|=|S_j|$. For every $1\le j \le p$ let $R_{S'_j}=\varphi(P_{S_j})$.
			
			Let $S':=\bigcap_{j=1}^p S'_j$.
			Note that \[\varphi(P_S) \subset \begin{cases}
				P_{S'}   & \text{ if $\varphi(H_0)=H_0$; } \\[3ex] 
				Q_{S'} & \text{ if $\varphi(H_0)=H_r$ and $k<m=n$; and  } \\[3ex] 
				\overline{P}_{S'}  & \text{ if $\varphi(H_0)=H_r$, $m<n$ and $k=\frac{m+n}{2}$. } 
			\end{cases}
			\] 
			It remains to show that  $|S'|=|S|$.
			
			For every pair of indices $l \neq t$, we have that $|S_l\triangle S_t|=2$. Thus,
			for every $A \in P_{S_l}\setminus P_{S_t}$ there exists a unique $A' \in P_{S_t}\setminus P_{S_l}$
			such that $|A\triangle A'|=2$. In particular, $d(A,A')=2$. 
			Suppose that $|S'_l\triangle S'_t|>2$. 
			
			Let now $A \in R_{S'_l}\setminus R_{S'_t}$. Since $|S'_l|=|S'_t|=|S|+1 \le k-1 \le n-2$, we have that either    
			$d(A,B)>2$  for every $B \in R_{S'_t}\setminus R_{S'_l}$ or $d(A,B_1)=d(A,B_2)=2$ for some two distinct 
			$B_1,B_2 \in R_{S'_t}\setminus R_{S'_l}$.
			This contradicts the fact that $\varphi$ is an automorphism. Thus, 
			\begin{equation}
				\label{eq:triangle_S'}
				|S'_l \triangle S'_t|=2,
			\end{equation}
			for every pair of distinct $S_l'$ and $S_t'$.

			We now claim that for every $S'_j$ distinct from $S_l'$ and $S_t'$ we have that
			\[S'_l\cap S'_t\subset S'_j.\] Suppose for a contradiction that there exists $S_j'$ 
			such that $S'_l\cap S'_t\not\subset S'_j.$ 
			Observe that there is a unique $p\in (S'_l\cap S'_t)\setminus S'_j$, as otherwise we would have $|S'_j\triangle S'_l|>2$
			and $|S'_j\triangle S'_t|>2$, contradicting~\eqref{eq:triangle_S'}. Further, we must have $S'_l\setminus S'_j=\{p\}$
			and $S'_t\setminus S'_j=\{p\}$, implying that $|S'_j\cup S'_l\cup S'_t|=|S'_j|+1=|S|+2$. 
			On the other hand, we have $|S_j\cup S_l\cup S_t|=|S|+3$.
			Thus, 
			\[|R_{S'_j}\cap R_{S'_l}\cap R_{S'_t}|=\binom{n-(|S|+2)}{k-(|S|+2)}=\binom{n-(|S|+2)}{n-k}\]
			and
			\[|P_{S_j}\cap P_{S_l}\cap P_{S_t}|=\binom{n-(|S|+3)}{k-(|S|+3)}=\binom{n-(|S|+3)}{n-k}.\]
			Since $\varphi$ is an automorphism, $\varphi$ sends $(P_{S_j}\cap P_{S_l}\cap P_{S_t})$ to $(R_{S'_j}\cap R_{S'_l}\cap R_{S'_t})$.
			By Pascal's Rule\footnote{$\binom{a}{b-1}+\binom{a}{b}=\binom{a+1}{b}$} we have that $\binom{n-(|S|+3)}{n-k-1}=0.$
			Therefore, $|S|>k-2$, a contradiction. 	Thus, $S'_l\cap S'_t\subset S'_j$, and $|S'|=|S|$ as claimed. 
		\end{itemize}
	\end{proof}
	For $0 \le s \le r$ and $1 \le i \le n$, let 
	\[X(s,i) :=\{A\in H_s:x_i\in A\} \textrm{ and } \overline{X}(s,i) :=\{A\in H_s:x_i\notin A\};\]
	and
	\[Y(s,i) :=\{A\in H_s:y_i\in A\} \textrm{ and } \overline{Y}(s,i) :=\{A\in H_s:y_i\notin A\}.\]
	Observe that when  $m < n$, we have  $X(s,i) = \emptyset$  for each  $i \in \{m+1, \dots, n\}$, 
	since there is no element  $x_i \in X$  (recall that  $|X| = m$). 
	
	\begin{lemma}	\label{lem:permutation_sigma} 
		Suppose that $k<m=n$ or $m<n$. Then there exists a $\pi_\varphi \in S_n$ such that, for every $1 \le i \le n$,
		\[\varphi(Y(0,i))=\begin{cases}
			Y(0,\pi_\varphi(i))   & \text{ if $\varphi(H_0)=H_0$; } \\[3ex] 
			X(r,\pi_\varphi(i))   & \text{ if $\varphi(H_0)=H_r$ and $k<m=n$; and  } \\[3ex] 
			\overline{Y}(r,\pi_\varphi(i))  & \text{ if $\varphi(H_0)=H_r$, $m<n$ and $k=\frac{m+n}{2}$. } 
		\end{cases}
		\] 
	\end{lemma}
	\begin{proof}
		By Lemma~\ref{lem:PQ},
		\[\varphi(Y(0,i))=\varphi(P_{\{y_i\}})=\begin{cases}
			P_{\{y_{\pi_{\varphi}(i)}\}}   & \text{ if $\varphi(H_0)=H_0$; } \\[3ex] 
			Q_{\{x_{\pi_{\varphi}(i)}\}} & \text{ if $\varphi(H_0)=H_r$ and $k<m=n$; and  } \\[3ex] 
			\overline{P}_{{\{y_{\pi_{\varphi}(i)}}\}}  & \text{ if $\varphi(H_0)=H_r$, $m<n$ and $k=\frac{m+n}{2}$; } 
		\end{cases}
		\] 
		The map  $\pi_\varphi$ is our desired permutation.
		See Figure~\ref{fig:Y_i's}, for an example. 
	\end{proof}
	From now on, $\pi_\varphi$ denotes the permutation given by Lemma~\ref{lem:permutation_sigma}.
	In what follows, we show that $\pi_\varphi$,  and whether $\varphi(H_0)=H_0$ or $\varphi(H_0)=H_r$,  
	determine the image of the vertices of $H_0$ under $\varphi$.

	\begin{figure}
		\centering
		\includegraphics[width=0.95\linewidth]{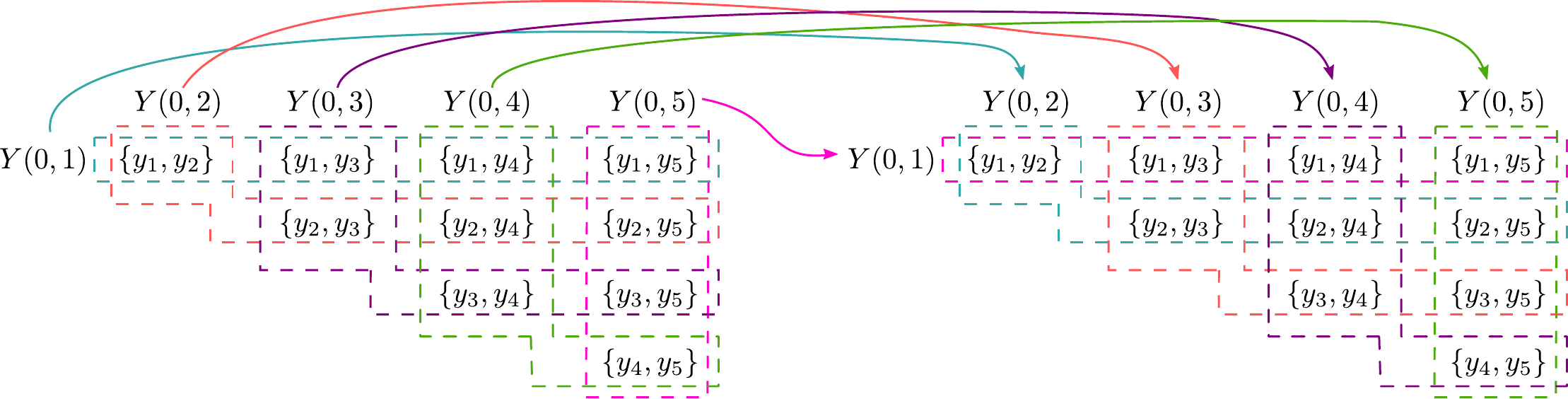}
		\caption{An example of $\pi_\varphi$ on sets $Y(0,1),\ldots,Y(0,n)$, where $n=5$ and $\varphi(H_0)=H_0$.}
		\label{fig:Y_i's}
	\end{figure}

	\begin{cor} 
		\label{cor:H_0-labelling}
		Let $A:=\{y_{t_1},\ldots,y_{t_k}\}\in H_0$.
		\begin{itemize}
			\item If $k=m=n$, then 
			\[\varphi(A)=\begin{cases}
				\{y_1,\ldots,y_n\} & \text{if $\varphi(H_0)=H_0$; and} \\
				\{x_1,\ldots,x_m\} & \text{if $\varphi(H_0)=H_r$.}
			\end{cases}
			\]
			\item If $k<m=n$ or $m<n$, then 
			\[\varphi(A)=\begin{cases} 
				\{y_{\pi_\varphi(t_1)},\ldots,y_{\pi_\varphi(t_k)}\} & \text{if $\varphi(H_0)=H_0$;} \\
				\{x_{\pi_\varphi(t_1)},\ldots,x_{\pi_\varphi(t_k)}\} & \text{if $\varphi(H_0)=H_r$ and $k<m=n$; and} \\
				V(K_{m,n})\setminus \{y_{\pi_\varphi(t_1)},\ldots,y_{\pi_\varphi(t_k)}\} & \text{if $\varphi(H_0)=H_r$, $m<n$ and $k=\frac{m+n}{2}$.}
			\end{cases}
			\]
		\end{itemize}
	\end{cor}
	\begin{proof}
		If $k=m=n$, then $|H_0|=|H_r|=1$ and
		\[\varphi (A)=\begin{cases}
			\{y_1,\ldots,y_n\} & \text{if $\varphi(H_0)=H_0$,} \\
			\{x_1,\ldots,x_m\} & \text{if $\varphi(H_0)=H_r$,}
		\end{cases}
		\]
		since $\{y_1,\ldots,y_n\}$ (resp. $\{x_1,\ldots,x_n\}$) is the only vertex in $H_0$ (resp. $H_r$).  
		
		Suppose that $k<m=n$ or $m<n$. Note that 
		$\{A \} =\bigcap\limits_{i=1}^kY(0,t_i).$  Lemma~\ref{lem:permutation_sigma} implies that: 
		\begin{itemize}
			\item if $\varphi(H_0)=H_0$, then $\{\varphi(A)\} = \bigcap\limits_{i=1}^kY(0,\pi_\varphi(t_i))$; 
			
			\item if  $\varphi(H_0)=H_r$ and  $k<m=n$, then $\{\varphi(A)\} = \bigcap\limits_{i=1}^k X(r,\pi_\varphi(t_i))$; and
			
			\item if  $\varphi(H_0)=H_r$, $m<n$ and $k=\frac{m+n}{2}$, then $\{\varphi(A)\} = \bigcap\limits_{i=1}^k \overline{Y}(r,\pi_\varphi(t_i))$.
		\end{itemize}
		Therefore,  
		\[\varphi(A)=\begin{cases} 
			\{y_{\pi_\varphi(t_1)},\ldots,y_{\pi_\varphi(t_k)}\} & \text{if $\varphi(H_0)=H_0$,} \\
			\{x_{\pi_\varphi(t_1)},\ldots,x_{\pi_\varphi(t_k)}\} & \text{if $\varphi(H_0)=H_r$ and $k<m=n$,} \\
			V(K_{m,n})\setminus \{y_{\pi_\varphi(t_1)},\ldots,y_{\pi_\varphi(t_k)}\} & \text{if $\varphi(H_0)=H_r$, $m<n$ and $k=\frac{m+n}{2}$.}
		\end{cases}
		\]
		
	\end{proof}

	\subsubsection{The action of $\varphi$ on $H_1$}
	We now characterize how $\varphi$ maps $H_1$ to $\varphi(H_1)$.
	Note that $\{X(1,1),\ldots,X(1,m)\}$ is a partition of $H_1$, and if $m=n$ (resp. $k=\frac{m+n}{2}$) 
	$\{Y(r-1,1),\ldots,Y(r-1,n)\}$ (resp.  $\{\overline{X}(r-1,1),\ldots,\overline{X}(r-1,m)\}$) is a partition of $H_{r-1}$.
	In the following result we characterize the behavior of $\varphi$ on the sets $X(1,1),\ldots,X(1,m)$. 
	Recall that we are assuming that $m\ge 2$ and $n>2$. We distinguish the following three cases:
	$m=2$; $k=m=n>2$; and, $m>2$ with either $k<m=n$ or $m<n$. We consider these cases separately.

	\begin{lemma}
		\label{lem:m=2}
		Suppose that $m=2$. Let $H_1=:\left \{A_1,\ldots,A_{\binom{n}{k-1}},B_1,\ldots,B_{\binom{n}{k-1}}\right \}$,
		so that $A_i \cap X=\{x_1\}$, $B_i \cap X=\{x_2\}$ and $A_i \cap Y=B_i \cap Y$,
		for all $1 \le i \le \binom{n}{k-1}$. Then there exist  $\alpha_\varphi \in 2^{\binom{Y}{k-1}}$, and  $\pi_\varphi'\in S_{\binom{n}{k-1}}$,
		such that the following hold.
		\begin{itemize}
			\item[$a)$] $\pi_\varphi'$ depends only on $\pi_\varphi$, and whether $\varphi(H_0)=H_0$ or $\varphi(H_0)=H_r$;
			
			\item[$b)$] $\varphi$ sends the pair $\left \{A_i,B_i\right \}$ to the pair
			$\left \{A_{\pi_\varphi'(i)},B_{\pi_\varphi'(i)} \right \}$; and
			
			\item[$c)$] for all $1 \le i \le \binom{n}{k-1}$,
			
			\[\varphi(A_i)=
			\begin{cases}
				B_{\pi_\varphi'(i)} & \text{ if } A_i \cap Y \in \alpha_\varphi, \textrm{ and } \\
				A_{\pi_\varphi'(i)} & \text{ otherwise. }
			\end{cases}\]
		\end{itemize}
	\end{lemma} 
	\begin{proof}
		Let $1 \le i< j \le \binom{n}{k-1}$. Note that $N(A_i)=N(B_i)$, while
		$N(A_i)\neq N(A_j)$, $N(B_i)\neq N(B_j)$ and $N(A_i)\neq N(B_j)$. Also, note that
		$\varphi(H_1)=H_1$. Let $C_{i_1},\ldots,C_{i_p}$ be the common neighbors of $A_i$ and $B_i$
		in $H_0$. Since $\varphi$ is an automorphism and
		$\varphi(H_1)=H_1$, there exist $1 \le i'\le \binom{n}{k-1}$ such that $A_{i'}$ and $B_{i'}$
		have $\varphi(C_{i_1}),\ldots,\varphi(C_{i_p})$ as common neighbors  in $\varphi(H_0)$.
		Therefore, $\varphi$ sends the pair $\{A_i,B_i\}$ to the pair $\{A_{i'},B_{i'}\}$.
		We set $\pi_\varphi'(i)=i'$. This proves $b)$.
		By Corollary~\ref{cor:H_0-labelling}, the values of $\varphi(C_{i_1}),\ldots,\varphi(C_{i_p})$ are  determined
		by $\pi_\varphi$ and  whether $\varphi(H_0)=H_0$ or $\varphi(H_0)=H_r$. Thus, $a)$ holds.
		Finally, to prove $c)$ we define $\alpha_\varphi$.
		If $\varphi(A_i)=A_{\pi_\varphi'(i)}$ and $\varphi(B_i)=B_{\pi_\varphi'(i)}$, then $A_i \cap Y \notin \alpha_\varphi$;
		and if $\varphi(A_i)=B_{\pi_\varphi'(i)}$ and $\varphi(B_i)=A_{\pi_\varphi'(i)}$, then $A_i \cap Y \in \alpha_\varphi$.
	\end{proof}
	
	\begin{lemma} \label{lem:m>2}
		Suppose that $m >2$. Then there exists $\sigma_\varphi \in S_m$, such that for every $1 \le i \le m$, the following hold.
		\begin{itemize}
			\item[$a)$] If $k<m=n$ or $m<n$, then 
			\[ \varphi(X(1,i))=\begin{cases}
				X(1,\sigma_\varphi(i)) & \textrm{ if } \varphi(H_0)=H_0,\\
				Y(r-1, \sigma_\varphi(i)) & \textrm{ if } \varphi(H_0)=H_r \textrm{ and }  k<m=n, \\
				\overline{X}(r-1,\sigma_\varphi(i)) & \textrm{ if }\varphi(H_0)=H_r, m < n \textrm{ and } k=\frac{m+n}{2}.
			\end{cases}
			\]
			
			\item[$b)$] If $k=m=n$, then
			\begin{itemize}
				\item $\varphi(X(1,i))$ is equal either to
				\[X(1,\sigma_\varphi(i)) \textrm{ or } \overline{Y}(1,\sigma_\varphi(i)) \textrm{ if } \varphi(H_0)=H_0, \]
				\[\overline{X}(r-1,\sigma_\varphi(i)) \textrm{ or } Y(r-1,\sigma_\varphi(i))  \textrm{ if } \varphi(H_0)=H_r.\]
				Moreover, exactly one of these two options holds for all $1 \le i \le m$.

				\item In addition there exists $\pi_\varphi \in S_n$ such that for every $1 \le j \le n$, 
				\[ \varphi(\overline{Y}(1,j))=\begin{cases}
					\overline{Y}(1,\pi_\varphi(j)) & \textrm{ if } \varphi(X(1,i))=X(1,\sigma_\varphi(i)) \textrm{ for all } 1 \le i \le m,\\
					X(1,\pi_\varphi(j)) & \textrm{ if } \varphi(X(1,i))=\overline{Y}(1,\sigma_\varphi(i)) \textrm{ for all } 1 \le i \le m,\\
					Y(r-1,\pi_\varphi(j)) & \textrm{ if } \varphi(X(1,i))=\overline{X}(r-1,\sigma_\varphi(i)) \textrm{ for all } 1 \le i \le m,\\
					\overline{X}(r-1,\pi_\varphi(j)) & \textrm{ if } \varphi(X(1,i))=Y(r-1,\sigma_\varphi(i)) \textrm{ for all } 1 \le i \le m.
				\end{cases}
				\]
			\end{itemize}
		\end{itemize}
	\end{lemma}
	\begin{proof}
		Let $A,B$ be distinct vertices either both in $H_1$ or both in $H_{r-1}$. If $A,B \in H_1$, then let $s=2$; otherwise, let $s=r-2$.
		First, we make some observations.
		\begin{itemize}
			\item Suppose that \[A \cap X \neq B \cap X \textrm{ and } A \cap Y \neq B \cap Y.\]
			\begin{itemize}
				\item If $|(A\cap X) \triangle (B \cap X)|>2$ or $|(A\cap Y) \triangle (B \cap Y)|>2$,
				then $A$ and $B$ have no common neighbor in $H_s$; and
				\item if $|(A\cap X) \triangle (B \cap X)|=2$ and $|(A\cap Y) \triangle (B \cap Y)|=2$,
				then $A$ and $B$ have exactly one common neighbor in $H_s$.
			\end{itemize}
			
			\item Suppose that \[A \cap X = B \cap X.\]
			\begin{itemize}
				\item If   $|(A\cap Y) \triangle (B \cap Y)|>2$,
				then $A$ and $B$ have no common neighbor in $H_s$.
				\item Suppose that $|(A\cap Y) \triangle (B \cap Y)|=2$. If $A,B \in H_1$, then
				$A$ and $B$ have exactly $m-1$ common neighbors in $H_2$; and if $A,B \in H_{r-1}$, then
				$A$ and $B$  have exactly $r-1$ common neighbors in $H_{r-2}$.
			\end{itemize}
			
			\item Suppose that \[A \cap Y = B \cap Y.\]
			\begin{itemize}
				\item If   $|(A\cap X) \triangle (B \cap X)|>2$,
				then $A$ and $B$ have no common neighbor in $H_s$.
				\item Suppose that $|(A\cap X) \triangle (B \cap X)|=2$. If $A,B \in H_1$, then
				$A$ and $B$ have exactly $k-1$ common neighbors in $H_2$; and if $A,B \in H_{r-1}$, then
				$A$ and $B$  have exactly $n-k+r-1$ common neighbors in $H_{r-2}$.
			\end{itemize}
		\end{itemize}
		
		Let  $A,B \in X(1,i)$ for some $1 \le i \le m$, with
		$|(A\cap Y) \triangle (B \cap Y)|=2$. Note that $A$ and $B$ have exactly $m-1$ common neighbors in $H_2$, and exactly one common neighbor in $H_0$.
		By the previous observations, we have that

		\begin{equation*} \varphi(A) \cap X =\varphi(B) \cap X \textrm{ or } \varphi(A) \cap Y = \varphi(B) \cap Y.  \end{equation*}
		
		Let $C \in X(1,i)$ such that $|(A\cap Y) \triangle (C \cap Y)|=2$ and $|(B\cap Y) \triangle (C \cap Y)|=2$.
		Suppose that \[\varphi(A) \cap X = \varphi(C) \cap X \textrm{ and } \varphi(B) \cap Y = \varphi(C) \cap Y.\]
		We have that $\varphi(A)\cap Y \neq \varphi(C)\cap Y=\varphi(B) \cap Y$ and $\varphi(B) \cap X \neq \varphi(C) \cap X =\varphi(A) \cap X.$
		If $\varphi(H_0)=H_0$ (resp. $\varphi(H_0)=H_r$), then  $\varphi(A)$ and $\varphi(B)$ have at most  one common neighbor in $H_2$ (resp. $H_{r-2}$) and $m \le 2$;
		which is a contradiction. Therefore,
		either \[\varphi(A) \cap X =\varphi(B) \cap X =\varphi(C) \cap X\]
		or \[\varphi(A) \cap Y =\varphi(B) \cap Y =\varphi(C) \cap Y.\]
		Let $A', B' \in \binom{Y}{k-1}$. It can be shown that there exists a sequence $A=C_1,\dots,C_l=B$ of elements in $\binom{Y}{k-1}$ such that for every
		$1 < j <l$, we have that $|C_{l-1}\triangle C_l|=|C_l \triangle C_{l+1}|=|C_{l-1} \triangle C_{l+1}|=2$.
		This implies that 
		\begin{equation}\label{eq:all} \varphi(A) \cap X =\varphi(D) \cap X \textrm{ for all } D \in X(1,i)
			\textrm{ or } 
			\varphi(A) \cap Y =\varphi(D) \cap Y \textrm{ for all } D \in X(1,i). 
		\end{equation}
		By similar arguments we have that if $k=m=n$ and $A \in \overline{Y}(1,j)$, then
		\begin{equation}\label{eq:all_y} \varphi(A) \cap X =\varphi(D) \cap X \textrm{ for all } D \in \overline{Y}(1,j)
			\textrm{ or } 
			\varphi(A) \cap Y =\varphi(D) \cap Y \textrm{ for all } D \in \overline{Y}(1,j).
		\end{equation}
		
		First we prove $a)$. Suppose that $k<m=n$ or $m<n$.
		\begin{itemize}
			\item Suppose that $\varphi(H_0)=H_0$. Suppose that $\varphi(A) \cap X\neq \varphi(B) \cap X.$
			Thus, $\varphi(A) \cap Y =\varphi(B) \cap Y$. We have that $\varphi(A)$ and $\varphi(B)$ 
			have exactly $k-1=m-1$ common neighbors in $H_{2}$, and exactly $n-(r-1)=1$ common neighbors in $H_{0}$. 
			This implies that $m=n=k$, a contradiction. By (\ref{eq:all}) we have that 
			\[\varphi(X(1,i))=X(1,\sigma_\varphi(i)),\]
			for some $1 \le \sigma_\varphi(i) \le m$.
			
			\item Suppose that $\varphi(H_0)=H_r$ and $k<m=n$.
			Suppose that $\varphi(A) \cap Y\neq \varphi(B) \cap Y.$
			Thus, $\varphi(A) \cap X =\varphi(B) \cap X$. We have that $\varphi(A)$ and $\varphi(B)$ 
			have exactly $r-1=m-1$ common neighbors in $H_{r-2}$. Thus, $r=m$, $r=\min\{m,k\}$ and $k<m$, a contradiction.
			By (\ref{eq:all}) we have that \[\varphi(X(1,i))= Y(r-1, \sigma_\varphi(i)),\]
			for some $1 \le \sigma_\varphi(i) \le m$.
			
			\item Suppose that $\varphi(H_0)=H_r$, $m< n$ and $k=\frac{m+n}{2}$.
			Suppose that $\varphi(A) \cap X\neq \varphi(B) \cap X.$
			Thus, $\varphi(A) \cap Y =\varphi(B) \cap Y$. We have that $\varphi(A)$ and $\varphi(B)$ 
			have exactly $n-k+r-1=m-1$ common neighbors in $H_{r-2}$, and exactly $k-(r-1)=1$ common neighbors in $H_{r}$. Thus, $n=m$, a contradiction.
			By  (\ref{eq:all}) we have that \[\varphi(X(1,i))= \overline{X}(r-1,\sigma_\varphi(i)),\]
			for some $1 \le \sigma_\varphi(i) \le m$.
		\end{itemize}
		
		Now we prove $b)$. Suppose that $k=m=n$. 
		By (\ref{eq:all}), we have that $\varphi(X(1,i))$  is equal to either
		\begin{equation} \label{eq:op}
			X(1,\sigma_\varphi(i)), \overline{Y}(1,\sigma_\varphi(i)), \overline{X}(r-1,\sigma_\varphi(i)) \textrm{ or } Y(r-1,\sigma_\varphi(i)),
		\end{equation}
		for some $1 \le \sigma_\varphi(i) \le m$. 
		For $1 \le i < j \le m$,
		we have  that
		\begin{equation}\label{eq:same}
			|X(1,i)\cap X(1,j)| = |\overline{X}(r-1,i)\cap \overline{X}(r-1,j)|= |\overline{Y}(1,i)\cap \overline{Y}(1,j)| = |Y(r-1,i)\cap Y(r-1,j)|=0,
		\end{equation}
		and
		\begin{equation}\label{eq:dif}
			|X(1,i)\cap \overline{Y}(1,j)| = |\overline{X}(r-1,i)\cap Y(r-1,j)|=1.
		\end{equation}
		Equalities (\ref{eq:same}) and (\ref{eq:dif}), and the fact that $\varphi(H_1)$ is equal either to $H_1$ or to $H_{r-1}$,  imply
		that exactly one of the options in (\ref{eq:op}) holds for all $1 \le i \le m$. 
		By (\ref{eq:all_y}) and equalities (\ref{eq:same}) and (\ref{eq:dif}), there exists $\pi_\varphi \in S_n$
		such that for every $1 \le j \le n$, 
		\[ \varphi(\overline{Y}(1,j))=\begin{cases}
			\overline{Y}(1,\pi_\varphi(j)) & \textrm{ if } \varphi(X(1,i))=X(1,\sigma_\varphi(i)) \textrm{ for all } 1 \le i \le m,\\
			X(1,\pi_\varphi(j)) & \textrm{ if } \varphi(X(1,i))=\overline{Y}(1,\sigma_\varphi(i)) \textrm{ for all } 1 \le i \le m,\\
			Y(r-1,\pi_\varphi(j)) & \textrm{ if } \varphi(X(1,i))=\overline{X}(r-1,\sigma_\varphi(i)) \textrm{ for all } 1 \le i \le m,\\
			\overline{X}(r-1,\pi_\varphi(j)) & \textrm{ if } \varphi(X(1,i))=Y(r-1,\sigma_\varphi(i)) \textrm{ for all } 1 \le i \le m.
		\end{cases}
		\]
	\end{proof}

	\begin{lemma} \label{lem:m>2_A}
		Suppose that $m >2$. If $k<m=n$ or $m<n$, then let $\pi_\varphi \in S_n$ be as in Lemma~\ref{lem:permutation_sigma}, and
		if $k=m=n$, then let $\pi_\varphi \in S_n$ be as in Lemma~\ref{lem:m>2}. Let $\sigma_\varphi$ be as in
		Lemma~\ref{lem:m>2}. Let $A \in H_1$.
		
		\begin{itemize}
			\item[$a)$] Suppose that $k<m=n$ or $m<n$. Then the value of $\varphi(A)$ depends only
			on $\pi_\varphi$, $\sigma_\varphi$ and whether $\varphi(H_0)$ is equal to $H_0$ or to $H_r$.
			
			\item[$b)$] Suppose that  $k=m=n$. Then the value of $\varphi(A)$ depends only
			on $\pi_\varphi$, $\sigma_\varphi$, whether $\varphi(H_0)$ is equal to $H_0$ or to $H_r$, and
			the two possibilities for $\varphi(X(1,1))$.
		\end{itemize}
	\end{lemma}
	\begin{proof}
		Let \[A \cap X=:\{x_i\}\ \textrm{ and } A \cap Y =:\{y_{t_1},\dots y_{t_{k-1}}\}. \]
		
		\begin{itemize} 
			
			\item[$a)$] Suppose that $k<m=n$ or $m<n$. 
			
			By Lemma~\ref{lem:m>2}, we have  that 
			\[\varphi(A) \cap X = \begin{cases} 
				\{x_{\sigma_\varphi(i)}\} & \textrm{ if } \varphi(H_0)=H_0, \\
				X \setminus \{x_{\sigma_\varphi(i)}\} & \textrm{ if }\varphi(H_0)=H_r, m < n \textrm{ and } k=\frac{m+n}{2},
			\end{cases}
			\]
			and 
			\[\varphi(A) \cap Y=\{y_{\sigma_\varphi(i)} \} \textrm{ if } \varphi(H_0)=H_r \textrm{ and }  k<m=n.\]
			Note that \[N(A)\cap H_0 = \left \{\{y_j,y_{t_1},\dots,y_{t_{k-1}}\}: j \in [m] \setminus\{t_1,\dots t_{k-1}\} \right \}.\]
			By Corollary~\ref{cor:H_0-labelling}, we have that $\varphi\left (N(A)\cap H_0 \right )$ is equal to
			\[\begin{cases}
				\left  \{ \{y_{\pi_\varphi(j)},y_{\pi_\varphi(t_1)},\ldots,y_{\pi_\varphi(t_{k-1})}\} :j \in [m] \setminus\{t_1,\dots t_{k-1} \} \right \} & \text{if $\varphi(H_0)=H_0$;} \\
				\left  \{ \{x_{\pi_\varphi(j)},x_{\pi_\varphi(t_1)},\ldots,x_{\pi_\varphi(t_{k-1})}\} :j \in [m] \setminus\{t_1,\dots t_{k-1} \} \right \}  & \text{if $\varphi(H_0)=H_r$  and $k<m=n$; and} \\
				Y \setminus \left  \{ \{y_{\pi_\varphi(j)},y_{\pi_\varphi(t_1)},\ldots,y_{\pi_\varphi(t_{k-1})}\} :j \in [m] \setminus\{t_1,\dots t_{k-1} \} \right \} & \text{if $\varphi(H_0)=H_r$, $m<n$ and $k=\frac{m+n}{2}$.}
				
			\end{cases}
			\]
			Therefore, 
			\[\varphi(A) \cap Y = \begin{cases} 
				\left \{y_{\pi_\varphi(t_1)},\ldots,y_{\pi_\varphi(t_{k-1})} \right \} & \textrm{ if } \varphi(H_0)=H_0, \\
				Y \setminus\left  \{y_{\pi_\varphi(t_1)},\ldots,y_{\pi_\varphi(t_{k-1})} \right \} & \textrm{ if }\varphi(H_0)=H_r, m < n \textrm{ and } k=\frac{m+n}{2},
			\end{cases}
			\]
			and 
			\[\varphi(A) \cap X= \{x_{\pi_\varphi(t_1)},\ldots,x_{\pi_\varphi(t_{k-1})}\} \textrm{ if } \varphi(H_0)=H_r \textrm{ and }  k<m=n.\]
			
			\item[$b)$] Suppose that  $k=m=n$.
			
			Let $\{j\}:=[m]\setminus \{t_1,\dots,t_{k-1}\}$.
			By Lemma~\ref{lem:m>2}, we have  that 
			\[\varphi(A) \cap X = \begin{cases} 
				\{x_{\sigma_\varphi(i)}\} & \textrm{ if } \varphi(H_0)=H_0 \textrm{ and } \varphi(X(1,1))=X(1,\sigma_\varphi(1)), \\
				\{x_{\pi_\varphi(j)}\} & \textrm{ if } \varphi(H_0)=H_0 \textrm{ and } \varphi(X(1,1))=\overline{Y}(1,\sigma_\varphi(1)), \\
				X \setminus \{x_{\sigma_\varphi(i)}\} & \textrm{ if } \varphi(H_0)=H_r \textrm{ and } \varphi(X(1,1))=\overline{X}(r-1,\sigma_\varphi(1)),\\
				X \setminus \{x_{\pi_\varphi(j)}\} & \textrm{ if } \varphi(H_0)=H_r \textrm{ and } \varphi(X(1,1))=Y(r-1,\sigma_\varphi(1)),
			\end{cases}
			\]
			and
			\[\varphi(A) \cap Y = \begin{cases} 
				Y \setminus \{y_{\pi_\varphi(j)}\} & \textrm{ if } \varphi(H_0)=H_0 \textrm{ and } \varphi(X(1,1))=X(1,\sigma_\varphi(1)), \\
				Y \setminus \{y_{\sigma_\varphi(i)}\} & \textrm{ if } \varphi(H_0)=H_0 \textrm{ and } \varphi(X(1,1))=\overline{Y}(1,\sigma_\varphi(1)), \\
				\{y_{\pi_\varphi(j)}\} & \textrm{ if } \varphi(H_0)=H_r \textrm{ and } \varphi(X(1,1))=\overline{X}(r-1,\sigma_\varphi(1)),\\
				\{y_{\sigma_\varphi(i)}\} & \textrm{ if } \varphi(H_0)=H_r \textrm{ and } \varphi(X(1,1))=Y(r-1,\sigma_\varphi(1)).
			\end{cases}
			\]
		\end{itemize}
	\end{proof}

	\subsubsection*{Exact upper bound for $|\aut(F_k(K_{m,n}))|$}
	
	First, we show that the value of $\varphi$ on $H_0$ and $H_1$ determine its value on 
	all the vertices of $F_k(K_{m,n})$. Afterwards, we use this information and the previous
	lemmas to obtain an upper bound on $|\aut(F_k(K_{m,n}))|$, that matches the lower bound
	implied by the subgroups of $\aut(F_k(K_{m,n}))$ shown in Section~\ref{sec:inc}. This completes
	the proof of Theorem~\ref{thm:bipartite}.

	\begin{lemma}\label{lem:H_i}
		Let $2 \le i \le r$. Then the images of the vertices in $H_i$ under $\varphi$ are determined by the images under $\varphi$
		of the vertices in $H_{i-1}$.
	\end{lemma}
	\begin{proof}
		We aim to show the following slightly stronger result: for any two distinct vertices $A, B\in H_i$, we have
		\begin{equation}
			\label{eq:neighborhoods}
			N(A)\cap H_{i-1}\ne N(B)\cap H_{i-1}.
		\end{equation}
		It is straightforward to see that if~\eqref{eq:neighborhoods} holds, then the lemma follows. 
		Observe that each vertex in $H_i$ has at least one neighbor in $H_{i-1}$. 
		Clearly, if $d(A,B)>2$, $N(A)\cap N(B)=\emptyset$, and then~\eqref{eq:neighborhoods} trivially holds. 
		Assume then that $d(A,B)=2$, or equivalently, that $|A\triangle B|=2$. Let $R:=A\cap B$, and let 
		$a,b\in V(K_{m,n})$ such that $A=R\cup \{a\}$ and $B=R\cup \{b\}$. Since $A,B\in H_i$, we have 
		either $\{a,b\}\subseteq X$ or $\{a,b\}\subseteq Y$. This observation and the facts that $i\ge 2$ and $k\le |Y|$
		together imply the existence of two vertices $x\in X\cap A\cap B$ and $y\in Y\setminus (A\cup B)$. 
		Let $A':=(A\setminus \{x\})\cup \{y\}$. Observe that $A'\in H_{i-1}$, and further, $A'\in N(A)$, 
		meanwhile $A'\notin N(B)$ because $A'\triangle B=\{a,b,x,y\}$. The result follows. 
	\end{proof}

	\begin{lemma}\label{lem:bipartite_upper_bound}
		Let $m \le n$ be positive integers with $(m,n) \neq (2,2)$. Then
		\ {}
		\begin{itemize}
			\item[a)] for all $n >2$ and $1 < k< n+1$,
			\[ |\aut(F_k(K_{2,n}))| \le \begin{cases} 
				\left |\mathbb{Z}_2 \wr_{\binom{[n]}{k-1}} S_n \right | & \textrm{ if } k\neq \frac{n+2}{2}, \\
				\left |  \mathbb{Z}_2 \wr_{\binom{[n]}{k-1}} (S_n\times \mathbb{Z}_2)  \right | & \textrm{ otherwise; and} \\
			\end{cases}
			\]

			\item[b)] for all $m \neq 2$,
			\[ |\aut(F_k(K_{m,n}))| \le \begin{cases}
				|\aut(K_{m,n})| & \textrm{ if } k\neq \frac{n+m}{2}; \\
				|\aut(K_{m,n}) \times \mathbb{Z}_2|  & \textrm{ otherwise.}
			\end{cases}\]

		\end{itemize}
	\end{lemma}
	\begin{proof}		
		Let $\varphi \in \aut(F_k(K_{m,n}))$. If $k<m=n$ or $m<n$, then let $\pi_\varphi \in S_n$ be as in Lemma~\ref{lem:permutation_sigma}, and
		if $k=m=n$, then let $\pi_\varphi \in S_n$ be as in Lemma~\ref{lem:m>2}. Let $\sigma_\varphi \in S_m$ be as in Lemma~\ref{lem:m>2}.
		Finally, if $m=2$, let $\alpha_\varphi \in 2^{\binom{Y}{k-1}}$ be as in Lemma~\ref{lem:m=2}.
		Note that there are: $n!$ possibilities for $\pi_\varphi$, $m!$ possibilities for $\sigma_\varphi$, and  $2^{\binom{n}{k-1}}$ possibilities
		for $\alpha_\varphi$. Moreover if $k\neq \frac{n+m}{2}$, then $\varphi(H_0)=H_0$; and
		if $k= \frac{n+m}{2}$, then $\varphi(H_0)=H_0$ or $\varphi(H_0)=H_r$. 
		
		\begin{itemize}
			
			\item[$a)$] Suppose that $m=2,n >2$ and $1 < k< n+1$. 
			
			By Corollary~\ref{cor:H_0-labelling}, the image
			under $\varphi$  of every vertex in $H_0$, depends only on $\pi_\varphi$ and whether $\varphi(H_0)$ equals $H_0$ or $H_2$. 
			By Lemma~\ref{lem:m=2}, the image under $\varphi$ of every vertex in $H_1$ depends on $\alpha_\varphi$, $\pi_\varphi$ and whether $\varphi(H_0)$ equals $H_0$ or $H_2$. 
			By Lemma~\ref{lem:H_i}, the image under $\varphi$ of every vertex in $H_2$, depends only on the images under $\varphi$
			of the vertices in $H_1$. Therefore,
			\[ |\aut(F_k(K_{2,n}))| \le \begin{cases} 
				n! \cdot 2^{\binom{n}{k-1}}=\left |\mathbb{Z}_2 \wr_{\binom{[n]}{k-1}} S_n \right | & \textrm{ if } k\neq \frac{n+2}{2}, \\
				n! \cdot 2 \cdot 2^{\binom{n}{k-1}}=\left |  \mathbb{Z}_2 \wr_{\binom{[n]}{k-1}} (S_n\times \mathbb{Z}_2) \right | & \textrm{ otherwise. } \\
			\end{cases}
			\]
			
			\item[$b)$] Suppose $m \neq 2$.
			
			Recall that 
			\[|\aut(K_{m,n})|=\begin{cases}
				\left  |S_m \times S_n \right |=m! \cdot n! & \textrm{ if } m \neq n,\\
				\left  |S_m \times S_n \rtimes \mathbb{Z}_2 \right |=m! \cdot n! \cdot 2 & \textrm{ if } m=n.
			\end{cases}\]
			
			By Corollary~\ref{cor:H_0-labelling} we have the following. If $k=m=n$ the image
			under $\varphi$  of every vertex in $H_0$, depends only on whether $\varphi(H_0)$ equals $H_0$ or $H_r$; and
			if $k<m=n$ or $m<n$ the image
			under $\varphi$  of every vertex in $H_0$, depends only on $\pi_\varphi$ and whether $\varphi(H_0)$ equals $H_0$ or $H_r$. 
			
			By Lemma~\ref{lem:m>2_A}, we have the following. If $k<m=n$ or $m<n$, 
			then the value of every vertex in $H_1$ depends only
			on $\pi_\varphi$, $\sigma_\varphi$ and whether $\varphi(H_0)$ is equal to $H_0$ or to $H_r$;
			if $k=m=n$, then the value of every vertex in $H_1$ depends only
			on $\pi_\varphi$, $\sigma_\varphi$, whether $\varphi(H_0)$ is equal to $H_0$ or to $H_r$, and
			the two possibilities for $\varphi(X(1,1))$. By Lemmas~\ref{lem:H_i} and Corollary~\ref{cor:H_0-labelling}, we have that 
			\[ |\aut(F_k(K_{m,n}))| \le \begin{cases}
				m! \cdot n!= |\aut(K_{m,n})| & \textrm{if } k\neq (n+m)/2 \textrm{ and } m<n , \\
				m! \cdot n! \cdot 2= |\aut(K_{m,n})| & \textrm{if } k\neq (n+m)/2 \textrm{ and } m=n , \\
				m! \cdot n! \cdot 2= |\aut(K_{m,n}) \times \mathbb{Z}_2| & \textrm{if } k= (n+m)/2   \textrm{ and } m<n, \\
				m! \cdot n! \cdot 2 \cdot 2= |\aut(K_{m,n}) \times \mathbb{Z}_2| & \textrm{if } k= (n+m)/2   \textrm{ and } m=n. \\
			\end{cases}\]
		\end{itemize}
	\end{proof}

	\paragraph{Acknowledgments } \-\ 
	
	We thank Carlos Hidalgo-Toscano and Irene Parada for various helpful discussions.

	\small \bibliographystyle{abbrv} 
	\bibliography{automorphisms}

\end{document}